\documentclass[a4paper]{article}

\usepackage[utf8]{inputenc}
\usepackage{hyperref}
\usepackage{amsmath,amssymb,amsthm}
\usepackage{tikz,graphics}
\usepackage{verbatim}

\usetikzlibrary{arrows,shapes}
\usetikzlibrary{trees}
\usetikzlibrary{matrix,arrows}
\usetikzlibrary{positioning}
\usetikzlibrary{calc,through}
\usetikzlibrary{decorations.pathreplacing}
\usepackage{pgffor}
\usetikzlibrary{decorations.pathmorphing}
\usetikzlibrary{decorations.markings}
\tikzset{snake it/.style={decorate, decoration={snake},draw}}
\tikzset{snake/.style={decorate,decoration={zigzag}}}

\newtheorem{theorem}{Theorem}[section]

\newtheorem{lemma}[theorem]{Lemma}
\newtheorem{proposition}[theorem]{Proposition}

\newtheorem{remark}[theorem]{Remark}

\numberwithin{equation}{section}

\title{Embedding perfectly balanced 2-caterpillar into its optimal hypercube}
\author{Rishikant Rajdeepak \\201521006@daiict.ac.in \and V. Sunita \\ v\_suni@daiict.ac.in}

\begin{document}
\maketitle

\begin{abstract}
A long-standing conjecture on spanning trees of a hypercube states that a balanced tree on $2^n$ vertices with maximum degree at most $3$ spans the hypercube of dimension $n$~\cite{havel1986}. In this paper, we settle the conjecture for a special family of binary trees. A $0$-caterpillar is a path. For $k\geq 1$, a $k$-caterpillar is a binary tree consisting of a path with $j$-caterpillars $(0\leq j\leq k-1)$ emanating from some of the vertices on the path. A $k$-caterpillar that contains a perfect matching is said to be perfectly balanced. In this paper, we show that a perfectly balanced $2$-caterpillar on $2^n$ vertices spans the hypercube of dimension $n$.
\end{abstract}

\section{Introduction}

Graph embeddings have significant applications in designing interconnect networks for high performance computing systems. Among the several embedding problems, mapping binary trees into hypercubes have wider applications because of the computational structure of trees and various properties of hypercubes~\cite{choudum2016}. 

In 1984, Havel~\cite{havel1984} conjectured that an equibipartite binary tree on $2^n$ ($n\geq 1$) vertices is a spanning tree of the $n$-dimensional hypercube. It attained attention of researchers after Havel and Liebl~\cite{havel1986} proved the result for equibipartite binary caterpillars, wherein a caterpillar is a tree such that if all leaves are removed then the remaining subgraph is a path. Such caterpillars are also called one-legged caterpillars where legs are its leaves. In~\cite{bezrukov1998} it is shown that binary caterpillars with each leg of same parity is a subgraph of its optimal hypercube, where, for a graph on $m$ ($2^{n-1}<m\leq 2^n$) vertices, the hypercube of dimension $n$ is its optimal hypercube. This was generalized in~\cite{monien2018} for equibipartite binary caterpillars with legs of arbitrary length.

In this paper, we discuss about a special family of binary trees known as $k$-caterpillars, $k\geq 0$. We show that if a $2$-caterpillar on $2^n$ vertices has a perfect matching then it spans the $n$-dimensional hypercube.

\section{Preliminaries}

\subsection{Definitions and Notations}
A $0$-caterpillar is a path. For $k\geq 1$, a $k$-caterpillar is a binary tree consisting of a path with $j$-caterpillars $(0\leq j\leq k-1)$ emanating from some of the vertices on the path. The path is called the backbone and its vertices the backbone vertices of the $k$-caterpillar. A leg of the $k$-caterpillar is a $j$-caterpillar, $0\leq j\leq k-1$, originating from a backbone vertex, including the backbone vertex (see Fig.~\ref{fig0}). 

\begin{figure}
\centering
{
\begin{tikzpicture}[scale=.6]

	\begin{scope}[every node/.style={circle,draw,fill,inner sep=1}]
		\node (1) at (0,0) [label=above: $1$] {};
		\node (2) at (2,0) [label=above: $2$] {};
		\node (3) at (4,0) [label=above: $3$] {};

		\node (11) at (0,-1) {};
		\node (12) at (0,-2)  {};
		\node (13) at (0,-3) {};
		\node (14) at (0,-4) [label=right: $N^1$] {};
		\node (121) at (-.5,-2) {};
		\node (122) at (-1,-2) {};
		\node (123) at (-1.5,-2) {};
		\node (141) at (-.5,-4) {};
		\node (142) at (-1,-4) {};
		
		\node (21) at (2,-1) {};
		\node (22) at (2,-2) [label=below: $N^2$] {};
		\node (211) at (1.5,-1) {};
		\node (212) at (1,-1) {};
		
		\node (31) at (4,-1) {};
		\node (32) at (4,-2) {};
		\node (33) at (4,-3) [label=right: $N^3$] {};
		\node (311) at (3.5,-1) {};
		\node (321) at (3.5,-2) {};
		\node (322) at (3,-2) {};
		%\draw[snake=zigzag,thick] (2) -- (-1,-1); 
	\end{scope}
		\draw[thick] (1) -- (3); \draw[thick] (1) -- (14); \draw[thick] (2) -- (22); \draw[thick] (3) -- (33); \draw[thick] (1) -- (14);\draw[thick] (12) -- (123); \draw[thick] (14) -- (142); \draw[thick] (21) -- (212); \draw[thick] (31) -- (311); \draw[thick] (32) -- (322);
\end{tikzpicture}
}
\caption{A $2$-caterpillar with a backbone of order $3$. It has three legs, each of which is a $1$-caterpillar.}
\label{fig0}
\end{figure}
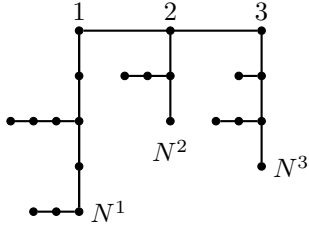

The order of a graph is its number of vertices, and its size is the number of edges. Let $C$ be a $k$-caterpillar on $m$ vertices and $N$ be the order of its backbone. We denote the $q^{th}$ leg of $C$ by $C^q$ and the order of the backbone of $C^q$ by $N^q$ (see Fig.~\ref{fig0}). Similarly, the $j^{th}$ leg in $C^q$ will be denoted by $C^{q,j}$ and the  order of the backbone of $C^{q,j}$ by $N^{q,j}$. 

\subsection{Properties of k-caterpillars and hypercubes}

A perfectly balanced graph is a graph with a perfect matching, i.e., the vertex set can be partitioned into pairs such that each pair is an edge. A tree has at most one perfect matching. A path of odd length is perfectly balanced. Deleting a non-matching edge in a perfectly balanced tree partitions the tree into two perfectly balanced subtrees. A path connecting two distinct vertices $x$ and $y$, denoted as $[x,y]$-path, is unique in a tree. 

A $k$-caterpillar is also a $j$-caterpillar for all $j\geq k$. A strictly $k$-caterpillar is a $k$-caterpillar which is not a $j$-caterpillar, for any $j\leq k-1$. A backbone of a $k$-caterpillar is not unique. It can be extended to another backbone of higher order by including the backbone of the first leg or that of the last leg (see Fig.~\ref{fig1}). Such extension reduces the order of the first leg or that of the last leg. The following result is imminent.

\begin{proposition}\label{prop:notunique}
If $C$ is a $k$-caterpillar, with $k\geq 1$, on $m$ vertices, then
\begin{enumerate}
	\item[(a)] a backbone of $C$ may not be unique,
	\item[(b)] there is a backbone  with the first leg and the last leg, each of order strictly less than $\frac{m}{2}\;(\mbox{for }m>2)$, 
	\item[(c)] if there is a backbone of order $2$ then $C$ is a $(k-1)$-caterpillar.
	
\end{enumerate}
\end{proposition}
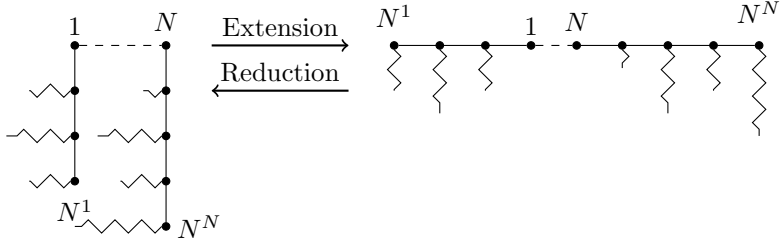
\begin{figure}[t]
\centering
\begin{tikzpicture}[scale=.6]
	\begin{scope}[every node/.style={circle,draw,fill,inner sep=1}]
		\node (1) at (0,0) [label=$1$] {};
		\node (2) at (2,0) [label=$N$] {};
		
		\node (11) at (0,-1) {};
		\node (12) at (0,-2) {};
		\node (13) at (0,-3) [label=below:$N^1$] {};
		
		\node (21) at (2,-1) {};
		\node (22) at (2,-2) {};
		\node (23) at (2,-3) {};
		\node (24) at (2,-4) [label=right:$N^N$] {};

		\draw[dashed] (1)--(2); \draw (1) -- (13); \draw (2)--(24);
		\draw[snake] (11) -- (-1,-1); \draw[snake] (12) -- (-1.5,-2); \draw[snake] (13) -- (-1,-3); \draw[snake] (21) -- (1.5,-1); \draw[snake] (22) -- (.5,-2); \draw[snake] (23) -- (1,-3); \draw[snake] (24) -- (0,-4);
	\end{scope}
	\begin{scope}[xshift=7cm,every node/.style={circle,draw,fill,inner sep=1}]
			\node (1) at (0,0) [label=$N^1$] {};
			\node (2) at (1,0) {};
			\node (3) at (2,0) {};
			\node (4) at (3,0) [label=$1$] {};
			\node (5) at (4,0) [label=$N$] {};
			\node (6) at (5,0) {};
			\node (7) at (6,0) {};
			\node (8) at (7,0) {};
			\node (9) at (8,0) [label=above:$N^N$] {};
		\draw (1) -- (4); \draw (5) -- (9); \draw[dashed] (4)--(5);
		\draw[snake] (1) -- (0,-1); \draw[snake] (2) -- (1,-1.5);\draw[snake] (3) -- (2,-1);\draw[snake] (6) -- (5,-.5);\draw[snake] (7) -- (6,-1.5);\draw[snake] (8) -- (7,-1);\draw[snake] (9) -- (8,-2);
	\end{scope}	
	\path[thick,->] (3,0) edge node[above] {Extension} (6,0);
	\path[thick,<-] (3,-1) edge node[above] {Reduction} (6,-1);
\end{tikzpicture}
\caption{Extension and reduction of a backbone in a $2$-caterpillar.}
\label{fig1}
\end{figure}

\begin{proof}
\begin{enumerate}
    \item[(a)] In Fig.~\ref{fig0}, the backbone of $C$ can be extended by including the backbone of its first leg. Thus, the backbone is not unique.
    \item[(b)] Let $B_0$ be a backbone of order $n_0$. Since $C$ is a $k$-caterpillar, its first leg $C^1$ is a $(k-1)$-caterpillar. If $C^1$ has one vertex, then part (b) of the proposition is true. If $C^1$ has more than one vertex, then by adding the backbone of $C^1$ to $B_0$, we obtain a backbone $B_1$ containing $B_0$. If $B_1$ has $n_1$ vertices, then $n_1>n_0$. Since $C^1$ is a $(k-1)$-caterpillar, the first leg of $C$ with the backbone $B_1$ is a $(k-2)$-caterpillar. By extending the backbone in this way, we obtain a sequence of backbones $B_0\subset B_1 \subset B_2\subset \dots$. Since, $C$ is finite there exists a maximal backbone $B_r$, for some $r\geq 1$. Clearly, the order of the first leg in $C$ with the backbone $B_r$ is one. Hence, part (b) of the proposition holds true.
    \item[(c)] Suppose $B$ is a backbone of $C$ of order 2, then by including the backbone of its only two legs, viz., $C^1$ and $C^2$, into $B$, we obtain a new backbone $B'$ of $C$. Since legs of $C^1$ and $C^2$ are $(k-2)$-caterpillars, the legs of $C$ with respect to the new backbone $B'$ are $(k-2)$-caterpillars. Therefore, $C$ is a $(k-1)$-caterpillar.
\end{enumerate}
\end{proof}

\begin{remark}
Consider a strictly 1-caterpillar $C$ on 6 vertices with the degree sequence (2,2,1,1,1,1). $C$ has a unique backbone of order 4, upto graph isomorphism.
\begin{figure}
\centering
\begin{tikzpicture}[scale=.6]
	\begin{scope}[every node/.style={circle,draw,fill,inner sep=1}]
		\node (1) at (0,0) [label=above:1] {};
		\node (2) at (1,0) [label=above:2] {};
		\node (21) at (1,-1) {};
		\node (3) at (2,0) [label=above:3] {};
		\node (31) at (2,-1) {};
		\node (4) at (3,0) [label=above:4] {};
    \end{scope}
    \draw (1)--(4); \draw (2) -- (21); \draw (3) -- (31);
\end{tikzpicture}
\caption{A strictly 1-caterpillar with unique backbone upto isomorphism.}\label{fig2}
\end{figure}
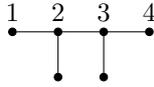
\end{remark}

A backbone of a $k$-caterpillar can also be reduced to a backbone of smaller order if there is a backbone vertex of degree $2$ and all legs before or after this vertex are at most $(k-2)$-caterpillars, in which case, the vertex becomes the first  or the last backbone vertex (see Fig.~\ref{fig1}). A strictly $2$-caterpillar has order at least $12$ (see Fig.~\ref{fig3} (a)), and if, in addition, it is perfectly balanced then the order is at least $16$ (see Fig.~\ref{fig3} (b)). 

\begin{figure}[t]
\centering
\begin{tikzpicture}[scale=.6]
	\begin{scope}[every node/.style={circle,draw,fill,inner sep=1}]
	\node (1) at (0,0) {};
	\node (2) at (2,0) {};
	\node (3) at (4,0) {};
	\node (12) at (0,-1) {};
	\node (13) at (0,-2) {};
	\node (122) at (-1,-1) {};
	\node (22) at (2,-1) {};
	\node (23) at (2,-2) {};
	\node (222) at (1,-1) {};
	\node (32) at (4,-1) {};
	\node (33) at (4,-2) {};
	\node (322) at (3,-1) {};
	
	\draw (1)--(3); \draw (1)--(13); \draw (2)--(23); \draw (3)--(33); \draw (12)--(122); \draw (22)--(222); \draw (32)--(322);
	
	\node[white] () at (2,-3.5) [label=(a)]{};
	\end{scope}
	
	\begin{scope}[xshift=8cm, every node/.style={circle,draw,fill,inner sep=1}]
	\node (1) at (0,0) {};
	\node (2) at (2,0) {};
	\node (3) at (4,0) {};
	\node (12) at (0,-1) {};
	\node (13) at (0,-2) {};
	\node (122) at (-1,-1) {};
	\node (22) at (2,-1) {};
	\node (23) at (2,-2) {};
	\node (222) at (1,-1) {};
	\node (32) at (4,-1) {};
	\node (33) at (4,-2) {};
	\node (322) at (3,-1) {};
	\node (1222) at (-1,-2) {};
	\node (2222) at (1,-2) {};
	\node (3222) at (3,-2) {};
	\node (34) at (5,-2) {};
	
	\draw (1)--(3); \draw (1)--(13); \draw (2)--(23); \draw (3)--(33); \draw (12)--(122); \draw (22)--(222); \draw (32)--(322); \draw[double] (122) -- (1222); \draw[double] (222) -- (2222); \draw[double] (322) -- (3222); \draw[double] (12) -- (13); \draw[double] (22) -- (23); \draw (32)--(33); \draw[double] (1) -- (2); \draw[double] (3) -- (32); \draw[double] (33) -- (34);
	
	\node[white] () at (2,-3.5) [label=(b)]{};
	\end{scope}
\end{tikzpicture}
\caption{(a) A strictly $2$-caterpillar and (b) a strictly perfectly balanced $2$-caterpillar. A matching edge is drawn as double line segment.}
\label{fig3}
\end{figure}
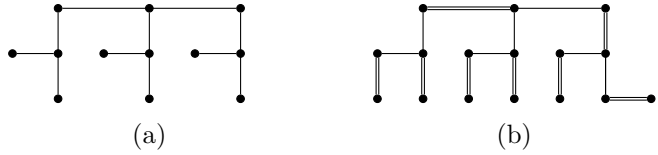

Consider a $k$-caterpillar on $m$ vertices with a backbone $B$ of order $N$. Suppose $x$ is the first backbone vertex and $y$ is any backbone vertex, we define $f_B(y)$ to be the order of the first $l$ legs of $C$, where $l$ is the order of the $[x,y]$-path. If the backbone vertices are labeled from $1$ to $N$, then we simply write $f_B(l)$. Clearly, $f_B$ is a strictly increasing function. We use this function to prove the following results. Unless explicitly specified, we assume that the backbone vertices are labeled from $1$ to $N$.

\begin{proposition}\label{prop:orderpartition}
Let $C$ be a $k$-caterpillar on $m>2$ vertices with a backbone $B$ of order $N$ such that $C^1$ and $C^N$ are both of order strictly less than $\frac{m}{2}$. Then, $\exists q$, with $1<q<N$, such that $f_B(q)\geq \frac{m}{2}$ and $f_B(q-1)<\frac{m}{2}$. 
\end{proposition}
\begin{proof}
Since $f_N$ is a strictly increasing function with $f_B(1)<\frac{m}{2}$ and $f_B(N)=m>\frac{m}{2}$, we get the required result.
\end{proof}

\begin{proposition}\label{prop:nonmatchingedge}
Let  $C$ be a perfectly balanced $k$-caterpillar on $m>2$ vertices with a backbone $B$ of order $N$. If there exists $q$, with $1<q<N$, such that  $f_B(q)> \frac{m}{2}$ and $f_B(q-1)<\frac{m}{2}$, then we can deduce that $(q-1,q)$ is not a matching edge. 
\end{proposition}
\begin{proof}
If $f_B(q)> \frac{m}{2}$ then $f_B(N)-f_B(q)< \frac{m}{2}$. So, if $(q-1,q)$ is not a matching edge, we are done, otherwise by reversing the labels of the backbone vertices, i.e. the backbone vertex $i\;(1\leq i\leq N)$ is relabeled by $N-i+1$, we get the required result. 
\end{proof}

A maximal backbone $B$ of a $k$-caterpillar $C$ is a backbone which can not be extended to a larger backbone of $C$ containing $B$. The first and the last legs in a maximal backbone are each of order one. If $C$ is of order $m=2$, then it has exactly two backbones, viz., (a) a backbone with one leg of order two and (b) a backbone with two legs, each of order one. In either case the first leg or the last leg can not have order strictly less than $\frac{m}{2}$, therefore, in the previous propositions we assumed $m>2$.

Suppose $q$ is a backbone vertex of a perfectly balanced $k$-caterpillar $C$. If the order of the leg $C^q$ is even then it is perfectly balanced, else $C^q\backslash\{q\}$, i.e., $q$ is removed from $C^q$, is  perfectly balanced. Moreover, if $C^q$ is of odd order then either $(q,q+1)$ or $(q-1,q)$ is a matching edge, in which case, $C^{q+1}$ or $C^{q-1}$ is of odd order. %We thus see that the number of legs of odd orders are even and they exist in pairs.

\begin{proposition}\label{prop:evenorderleg}
Let $C$ be a perfectly balanced $k$-caterpillar with $N$ backbone vertices and $M$ be its perfect matching. Suppose $(i,j)\in  M$ is a backbone edge with $1\leq i<j<N$, then, either the first leg of even order lies at an odd distance from $j$ or $N-1$ is at odd distance from $j$ with $(N-1,N)\in  M$.
\end{proposition}
\begin{proof}
Let $C^q$ be the first leg of even order from $j$, i.e., $q>j$ is the minimum integer for which $C^q$ is of even order. Then $C^q$ is perfectly balanced and alternate edges on $[i,q]$-path are matching edges, with $(q-1,q)\notin M$, so $q-j$ is odd (see Fig.~\ref{fig4}). If no such path exists then $N-i$ is odd with $(N-1,N)\in M$.
\end{proof}

\begin{figure}[t]
\centering
\begin{tikzpicture}[scale=.6]

	\begin{scope}[every node/.style={circle,draw,fill,inner sep=1}]
		
		\node (1) at (1,0) [label=above: $C^1$] {};
		\node (i) at (3,0) [label=above: $C^i$] {};
		\node (j) at (4,0) [label=above: $C^j$] {};
		\node (k) at (6,0) [label=above: $C^q$] {};
		\node (k1) at (6,-1) {};
		
		\node[white] (kN) at (6,-3) {};
		\node (N) at (8,0) [label=above: $C^N$]{};
	\end{scope}
		\draw[dashed] (1)--(i); \draw[double] (i) -- (j);\draw[dashed] (j)--(k); \draw[dashed] (k) -- (N);
		\draw[double] (k)--(k1); 
		\draw[snake it] (1) -- (1,-1.5); \draw[snake it] (i) -- (3,-2.5); \draw[snake it] (j) -- (4,-2);\draw[snake it] (k1) -- (kN);  \draw[snake it] (N) -- (8,-2.5);
\end{tikzpicture}
\caption{$(i,j)$ is a matching edge drawn as double line segment. $C^q$ is the first leg of odd length from the backbone vertex $j$.}
\label{fig4}
\end{figure}
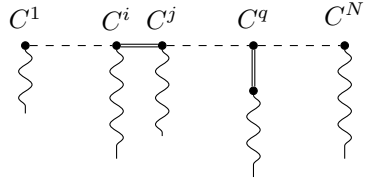

\begin{proposition}\label{prop:1cat_evenodd}
Let $C$ be a perfectly balanced $1$-caterpillar on $m\;(m\geq 2)$ vertices with a backbone $B$ of order $N$. Let $x$ be the first backbone vertex on $B$, then there exist backbones $B'$ and $B''$ of orders $N'$ and $N''$, respectively, both having $x$ as its first backbone vertex, such that
\begin{itemize}
	\item[(a)] the $[x,N']$-path is of even length and,
	\item[(b)] the $[x,N'']$-path is of odd length.
\end{itemize}
\end{proposition}
\begin{proof}
The result can be obtained by extending and reducing the backbone $B$. Suppose the order of $B$ is odd. If the order of the last leg $C^N$ is greater than $1$, then $B$ can be extended to a backbone of even order by including the second vertex on $C^N$. If the order of $C^N$ is one, then the order of $C^{N-1}$ is odd. Exchange the edge $(N-1,N)$ by the path $C^{N-1}$ to form a new backbone of even order. By the similar approach we get a backbone of odd order, if $B$ were of even order.
\end{proof}

\begin{remark}
We can apply Propositon~\ref{prop:1cat_evenodd} on a leg of a $2$-caterpillar. However, if a leg is of order $2$, then the leg has unique backbone, which is of odd length.
\end{remark}

A hypercube $\mathcal{Q}_n$ of dimension $n$ is a graph with the vertex set $\mathbb{Z}_2^n$ and two vertices being adjacent if and only if the Hamming distance between them is exactly one. We use the following properties of the hypercube in this paper.
\begin{lemma}~\cite{choudum2016}\label{lem:choudum2016}
A hypercube of dimension $n\geq 1$, is
\begin{enumerate}
	\item[1.] $K_1$ symmetric, i.e., vertex symmetric,
	\item[2.] $K_2$ symmetric, i.e., edge symmetric,
	\item[3.] $K_{1,2}$ symmetric for $n\geq 2$, i.e., $P_3$ symmetric,
	\item[4.] $K_{1,3}$ symmetric for $n\geq 3$, i.e., claw symmetric
	\item[5.] $C_4$ symmetric for $n\geq 2$.
\end{enumerate}
where $K_n$ is a complete graph, $P_n$ is a path and $C_n$ is a cycle, all on $n$ vertices, and $K_{p,q}$ is a complete bi-partite graph with parts of order $p$ and $q$.
\end{lemma}

\section{The Embedding}

An embedding is an injective graph homomorphism. In this paper, we show that any perfectly balanced $2$-caterpillar on $m$ vertices, where $2^{n-1}<m\leq 2^n$ ($n\geq 1)$, is embeddable into a hypercube of dimension $n$. It is sufficient to show that a perfectly balanced $2$-caterpillar on $2^n$ vertices span the hypercube of dimension $n$. 
%In this paper, we often use the  notation $\langle X\rangle$ to represent the subgraph, in a graph $G$, induced by the vertices in $X$, where $X$ is a collection of subgraphs of $G$. 

\begin{theorem}
Let $C$ be a perfectly balanced $2$-caterpillar on $2^n\;(n\geq 1)$ vertices. Then $C$ is a subgraph of the hypercube $\mathcal{Q}_n$ of dimension $n$.
\end{theorem}
\begin{proof}
For $n\leq 3$, there are exactly $6$ perfectly balanced $2$-caterpillars, viz., two perfectly balanced $0$-caterpillars on $2$ and $4$ vertices, and four perfectly balanced $1$-caterpillars on $8$ vertices. Each of them are embeddable into the respective optimal hypercubes. For $n=4$, we have, using the brute force method, verified that the theorem holds. For $n\geq 5$, we prove a stronger result as given in the following theorem.
\end{proof}

The following theorem states that a perfectly balanced $2$-caterpillar, on at least $32$ vertices, is embeddable into its optimal hypercube, with at most four fixed vertices being mapped to some fixed graph patterns, where $K_1,\; K_2,\; K_{1,2}$, $K_{1,3}$ and $C_4$ are among the fixed graph patterns.

\begin{theorem}
For $n\geq 5$, let $C$ be a perfectly balanced $2$-caterpillar on $2^n\;(n\geq 1)$ vertices with a backbone $B$ on $N$ vertices. Let $x$ be the first backbone vertex on $C$, $y$ be the $j^{th}\;(2\leq j\leq N^1)$ backbone vertex on $C^1$, $z$ be the end vertex on the path $C^{1,j}$ and $\alpha$ be the $(j-1)^{th}$ vertex, if it exists, on $C^1$ (see Fig.~\ref{fig5} (a)). Then, there exists an embedding $\phi$ of $C$ into $\mathcal{Q}_n$ such that $\phi(\{x,y,z\})$ induces some fixed graph patterns (see Fig.~\ref{fig5} (b)), viz.,
\begin{enumerate}
\item[1.] if $[x,y]$-path is of odd length and
		\begin{itemize}
			\item[(i)] if $[y,z]$-path is of odd length then the sequence $[\phi(x),\phi(y),\phi(z)]$ is $P_3$.
			\item[(ii)] else, if $[y,z]$-path is of even length then the sequence $[\phi(y),\phi(x),\phi(z)]$ is $P_3$.
		\end{itemize}

\item[2.] else, $[x,y]$-path is of even length, in which case,
		\begin{itemize}
			\item[(i)] if $[y,z]$-path is of even length then $\{\phi(x),\phi(y),\phi(\alpha),\phi(z)\}$ induces a claw $K_{1,3}$, with $\phi(\alpha)$ as its central vertex.
			\item[(ii)] else if $[y,z]$-path is of odd length then the sequence  $[\phi(x),\phi(\alpha),\phi(y),\linebreak\phi(z),\phi(x)]$ is $C_4$.
		\end{itemize}
\end{enumerate}
\end{theorem}

\begin{proof}
We prove by induction on $n$. For the base case of $n=5$, the result can be verified by the brute force method. We now proceed with the induction step.

Assume that any perfectly balanced $2$-caterpillar on $2^n\;(n\geq 5)$ vertices is embeddable into $\mathcal{Q}_n$ with an embedding satisfying the conditions $1$ or $2$, as given in the theorem. Let $C$ be a perfectly balanced $2$-caterpillar on $2^{n+1}$ vertices. Let $M$ be the perfect matching of $C$. Without loss of generality, assume that the order of the backbone of $C$ is $N\geq 3$ and the order of $C^1$ and $C^N$ are both strictly less than $2^n$ (by Proposition 1). The proof that $C$ has an embedding $\phi$ into $\mathcal{Q}_{n+1}$, as required, is exhibited by performing the following two steps.

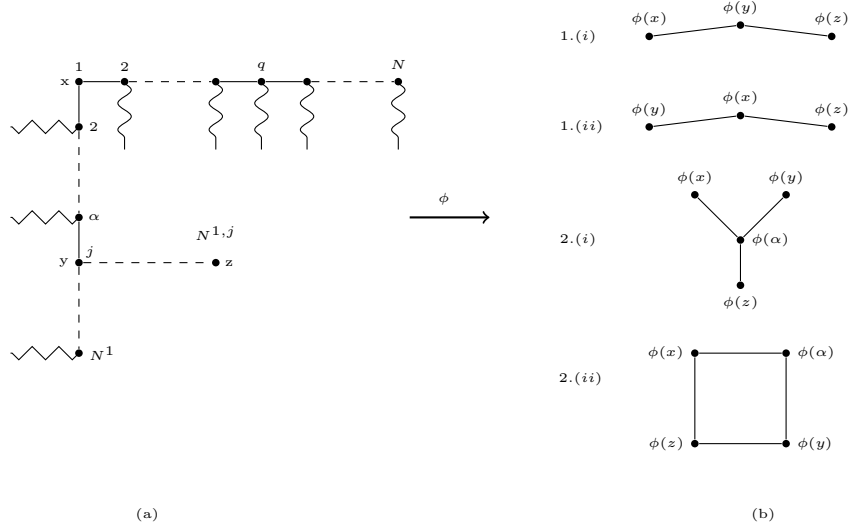
\begin{figure}
\centering
\tiny
\begin{tikzpicture}[scale=.3]
	\begin{scope}[every node/.style={circle,fill,inner sep=1}]
		\node (1) at (0,0) [label=$1$,label=left:x]{};
		\node (2) at (2,0) [label=$2$] {};
		\node (q-1) at (6,0) {};
		\node (q) at (8,0) [label=$q$] {};
		\node (q+1) at (10,0) {};
		\node (N) at (14,0) [label=$N$] {};
		\node (21) at (0,-2) [label=right:$2$] {};
		\node (j-1) at (0,-6) [label=right:$\alpha$] {};
		\node (j) at (0,-8) [label=above right: $j$,label=left:y] {};
		\node (N1) at (0,-12) [label=right: $N^1$] {};
		\node (N1j) at (6,-8) [label=above: $N^{1,j}$,label=right: z] {};	
		
		\draw (1) -- (2); \draw[dashed] (2) -- (q-1);\draw (q-1)--(q)--(q+1); \draw[dashed] (q+1)--(N);
		\draw (1)--(21); \draw[dashed] (21)--(j-1); \draw (j-1)--(j); \draw[dashed] (j)--(N1j); \draw[dashed] (j)--(N1);
		
		\draw[snake it] (2,0) -- (2,-3); \draw[snake it] (10,0) -- (10,-3); \draw[snake it] (8,0) -- (8,-3);\draw[snake it] (14,0) -- (14,-3); \draw[snake it] (6,0) -- (6,-3); \draw[snake] (0,-2) -- (-3,-2); \draw[snake] (0,-6) -- (-3,-6); \draw[snake] (0,-12) -- (-3,-12);
		
	\end{scope}
	
	\draw[->,thick] (14.5,-6) -- (18,-6);
	\node () at (16,-6) [label=above:$\phi$] {};
	
	\begin{scope}[xshift=24cm,yshift=2cm]
		\node[circle,fill,inner sep=1] (x) at (1,0) [label=$\phi(x)$] {};
		\node[circle,fill,inner sep=1] (y) at (5,.5) [label=$\phi(y)$] {};
		\node[circle,fill,inner sep=1] (z) at (9,0) [label=$\phi(z)$] {};
		
		\node at (-3.5,0) [label=right:$1. (i)$] {};
		\draw (x)--(y)--(z);
	\end{scope}

	\begin{scope}[xshift=24cm,yshift=-2cm]
		\node[circle,fill,inner sep=1] (y) at (1,0) [label=$\phi(y)$] {};
		\node[circle,fill,inner sep=1] (x) at (5,.5) [label=$\phi(x)$] {};
		\node[circle,fill,inner sep=1] (z) at (9,0) [label=$\phi(z)$] {};
		
		\node at (-3.5,0) [label=right:$1. (ii)$] {};
		\draw (y)--(x)--(z);
		
	\end{scope}
	
	\begin{scope}[xshift=28cm,yshift=-7cm]
		\node[circle,fill,inner sep=1] (alpha) at (1,0) [label=right:$\phi(\alpha)$] {};
		\node[circle,fill,inner sep=1] (x) at (-1,2) [label=$\phi(x)$] {};
		\node[circle,fill,inner sep=1] (y) at (3,2) [label=$\phi(y)$] {};
		\node[circle,fill,inner sep=1] (z) at (1,-2) [label=below:$\phi(z)$] {};
		
		\node at (-7.5,0) [label=right:$2. (i)$] {};
		\draw (x)--(alpha)--(z); \draw (alpha)--(y);
		
	\end{scope}
	\begin{scope}[xshift=25.5cm,yshift=-12cm]
		\node[circle,fill,inner sep=1] (x) at (1.5,0) [label=left:$\phi(x)$] {};
		\node[circle,fill,inner sep=1] (alpha) at (5.5,0) [label=right:$\phi(\alpha)$] {};
		\node[circle,fill,inner sep=1] (y) at (5.5,-4) [label=right:$\phi(y)$] {};
		\node[circle,fill,inner sep=1] (z) at (1.5,-4) [label=left:$\phi(z)$] {};
		
		\node at (-3.5,-2) [label=$2. (ii)$] {};
		\draw (x)--(alpha)--(y)--(z)--(x);
	\end{scope}
	\node[white] () at (3,-20) [label=(a)] {};	
	\node[white] () at (30,-20) [label=(b)] {};
\end{tikzpicture}
\caption{(a) A $2$-caterpillar with three fixed vertices $x,y$ and $z$ on its first leg and (b) An embedding $\phi$ mapping $x,y$ and $z$ into some fixed patterns in $\mathcal{Q}_n$.}
\label{fig5}
\end{figure}

\begin{enumerate}
\itemsep0.5em
\item[I.] \textit{Partition of the $2$-caterpillar:} $C$ is partitioned into at most four subtrees, say $X,X_2,Y_2$ and $Z_2$, such that if $\{(x_1,x_2),(y_1,y_2),(z_1,z_2)\}$ were the edges deleted then, as seen in Fig.~\ref{fig6},
\begin{enumerate}
	\item[(a)] $(x_1,x_2)$ lie on the backbone of $C$, $(y_1,y_2)$ lie on the backbone of the leg $C^{x_1}$ and $(z_1,z_2)$ lie on the leg $C^{x_1,y_1}$, and
	\item[(b)] $x_1,y_1$ and $z_1$ are contained in $X$ and $x_2,y_2$ and $z_2$ are contained in $X_2,Y_2$ and $Z_2$, respectively,
\end{enumerate}
such that $X$ is perfectly balanced $2$-caterpillar of order $2$. The remaining subtrees, i.e. $X_2,Y_2$ and $Z_2$, are joined by some new edges to form a perfectly balanced $2$-caterpillar, say $Y$, of order $2^n$, such that $\{x_2,y_2,z_2\}$ lie on one of the fixed patterns.

\item[II.] \textit{Extension of embeddings:} By the induction hypothesis, there exists an embedding $\phi_1:X\rightarrow \mathcal{Q}_n$ such that $\phi_1(\{x_1,y_1,z_1\})$ lie on one of the fixed patterns satisfying one of the four conditions, as mentioned in the theorem. By construction, $\{x_2,y_2,z_2\}$ lie on one of the fixed pattern, so any embedding $\phi_2:Y\rightarrow \mathcal{Q}_n$, which exists by the induction hypothesis, will preserve the pattern. By Lemma~\ref{lem:choudum2016}, there exists automorphisms $\pi_1$ and $\pi_2$ on $\mathcal{Q}_n$ such that 
$\pi_1\circ\phi_1(x_1)=\pi_2\circ\phi_2(x_2)$, $\pi_1\circ\phi_1(y_1)=\pi_2\circ\phi_2(y_2)$ and $\pi_1\circ\phi_1(z_1)=\pi_2\circ\phi_2(z_2)$. Define an embedding $\phi:C\rightarrow \mathcal{Q}_{n+1}$ by
\begin{equation}\label{eq:phi}
\phi(x)=\begin{cases} 0\pi_1\circ\phi_1(x);\; \mbox{ if } x\in X, \\ 1\pi_2\circ\phi_2(x);\; \mbox{ if } x\in Y. \end{cases}
\end{equation}
\end{enumerate}

\begin{figure}
		\centering
		
			\begin{tikzpicture}[scale=0.5]
				\node[circle,fill,draw,inner sep=1] (1) at (-1,0) [label=$1$] {};
				\node[circle,fill,draw,inner sep=1] (q-1) at (1,0) [label=$x_2$] {};
				\node[circle,fill,draw,inner sep=1] (q) at (2,0) [label=$x_1$] {};
				\node[circle,fill,draw,inner sep=1] (N) at (5,0) [label=$N$] {};
				\node[circle,fill,draw,inner sep=1] (y1) at (2,-3) [label=right:$y_1$] {};
				\node[circle,fill,draw,inner sep=1] (y2) at (2,-2) {};
				\node[circle,fill,draw,inner sep=1] (y3) at (1,-3) {};
				\node[circle,fill,draw,inner sep=1] (y4) at (2,-4) [label=below right:$y_2$] {};
				\node[circle,fill,draw,inner sep=1] (Nq) at (2,-6) [label=below:$N^{x_1}$] {};
				\node[circle,fill,draw,inner sep=1] (z1) at (-1,-3) [label=below:$z_1$] {};
				\node[circle,fill,draw,inner sep=1] (z2) at (-2,-3) [label=below:$z_2$] {};
				\node[circle,fill,draw,inner sep=1] (z3) at (-4,-3) {};
				\node[circle,fill,draw,inner sep=1] (x3) at (1,-2) [label=below:$x'$] {};
				
				\draw[dashed] (1)--(q-1); \draw (q-1) to node {/} (q); \draw[dashed] (q) -- (N);
				\draw[dashed] (q) -- (y2); \draw[dashed] (y3) -- (z1); \draw[dashed] (z2) -- (z3);
				\draw (z1) to node {/} (z2); \draw[dashed] (y4) -- (Nq);
				\draw (y1) -- (y2); \draw (y1) -- (y3); \draw (y1) to node {---} (y4);
				
				\draw[snake it] (1) -- (-1,-2); \draw[snake it] (q-1) -- (x3); \draw[snake it] (N) -- (5,-2);
				\draw[snake] (y2) -- (4,-2); \draw[snake] (y4) -- (4,-4); \draw[snake] (Nq) -- (4,-6);
			\end{tikzpicture}		
			\caption{Partition of a perfectly balanced $2$-caterpillar.}
			\label{fig6}
		\end{figure}
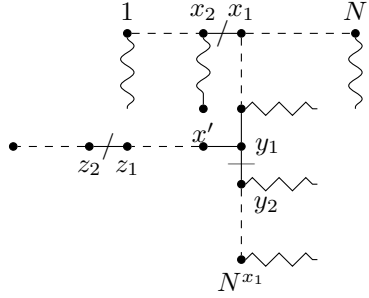
It then follows that $\phi$ is an embedding and $(\phi(x_1),\phi(x_2))$, $ (\phi(y_1),\phi(y_2))$ and $(\phi(z_1),\phi(z_2))$ form edges in $\mathcal{Q}_{n+1}$. Thus, once a partition, $\{X,Y\}$ of $C$, is obtained, embeddings $\phi_1$ and $\phi_2$ exist by the induction hypothesis. So, we only need to show that $C$ can be partitioned as required in step I. 

We adopt some notations to be used in the proof. The order of a graph $G$ is denoted by $o(G)$. A subgraph induced by $X$, where $X$ is a subgraph of $G$, is denoted by $\langle X\rangle$. Recall that $f_B(y)$ is the order of the first $l$ legs of a $k$-caterpillar with a backbone $B$ of $N$ vertices, where $y$ is the $l$-th backbone vertex. We discuss the proof in cases, as follows. 

By Proposition~\ref{prop:orderpartition}, there exists an integer q, with $1<q<N$, such that $f_B(q)\geq 2^n$ and $f_B(q-1)<2^n$. As seen in Fig.~\ref{fig6}, the value of $f_B(q)$ determines the following three cases.
\begin{enumerate}
\itemsep0.5em
\item[1.] If $f_B(q)=2^n$, then $(q,q+1)\notin M$. Put $x_1=q+1$ and $x_2=q$. Delete $(x_1,x_2)$ to get $X=\langle C^{x_1},\dots,C^N\rangle$ and $Y=\langle C^1,\dots,C^{x_2}\rangle$. By the induction hypothesis, $X$ and $Y$ can be embedded into $\mathcal{Q}_n$ via maps $\phi_1$ and $\phi_2$, respectively. By vertex-symmetry of $\mathcal{Q}_n$, we get $\phi_1(x_1)=\phi_2(x_2)$. Hence, the extended embedding $\phi$, as defined in Eq.~\ref{eq:phi}, maps $\{x_1,x_2\}$ into an edge $(\phi(x_1),\phi(x_2))$ of $\mathcal{Q}_{n+1}$.

\item[2.] If $f_B(q)=2^n+1$, then $(q,q+1)\in M$. Put $x_1=q$ and $x_2=q+1$. Let $C^{x_1,y_1}$ be the first path of odd length from $x_1$, on the leg $C^{x_1}$, and $(z_1,z_2)$ be the last edge on this path (see Fig.~\ref{fig7}). Then, $[x_1,z_1]$-path is of odd length. Delete $(x_1,x_2)$ and $(z_1,z_2)$ to  obtain $X=\langle C^1,\dots,C^{q-1},C^{x_1}\backslash z_2\rangle$. All the matching edges edges along the $[x_1,z_1]$-path in $C$  become non-matching edges in $X$ and vice-versa. Add $(x_2,z_2)$, which becomes a new matching edge, to obtain $Y=\langle (x_2,z_2),C^{x_2},\dots,C^N\rangle$. By the induction hypothesis $1.(ii)$, we get $(\phi_1(x_1),\phi_1(z_1))$ as an edge in $\mathcal{Q}_n$. By construction $(x_2,z_2)$ is an edge in $Y$, so $(\phi_2(x_2),\phi_2(z_2))$ is an edge in $\mathcal{Q}_n$. By edge-symmetry of $\mathcal{Q}_n$, we get $\phi_1(x_1)=\phi_2(x_2)$ and $\phi_1(z_1)=\phi_2(z_2)$. Thus, $(\phi(x_1),\phi(x_2))$ and $(\phi(z_1),\phi(z_2))$ form edges in $\mathcal{Q}_{n+1}$, via the map $\phi$.

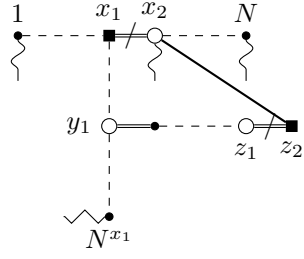
\begin{figure}
\centering
    \begin{tikzpicture}[scale=.6]
    	\node[circle,fill,draw,inner sep=1] (1) at (0,0) [label=$1$] {};
    	\node[fill,draw,inner sep=2] (q) at (2,0) [label=$x_1$] {};
    	\node[circle,draw,inner sep=2] (q+1) at (3,0) [label=$x_2$] {};
    	\node[circle,fill,draw,inner sep=1] (N) at (5,0) [label=$N$] {};
    	\node[circle,draw,inner sep=2] (y1) at (2,-2) [label=left:$y_1$] {};
    	\node[circle,fill,draw,inner sep=1] (Nq) at (2,-4) [label=below:$N^{x_1}$] {};
    	\node[circle,fill,draw,inner sep=1] (y2) at (3,-2) {};
    	\node[circle,draw,inner sep=2] (z1) at (5,-2) [label=below:$z_1$] {};
    	\node[fill,draw,inner sep=2] (z2) at (6,-2) [label=below:$z_2$] {};
    	
    	\draw[dashed] (1)--(q); \draw[double] (q) to node {/} (q+1); \draw[dashed] (q+1) -- (N);
    	\draw[dashed] (q) -- (y1); \draw[double] (y1) -- (y2); \draw[dashed] (y2) -- (z1);
    	\draw[double] (z1) to node {/} (z2); \draw[dashed] (y1) -- (Nq);
    	
    	\draw[snake it] (1) -- (0,-1); \draw[snake it] (q+1) -- (3,-1); \draw[snake it] (N) -- (5,-1); \draw[snake] (Nq) -- (1,-4);
    	
    	\draw[thick] (z2) edge (q+1);
    \end{tikzpicture}		
    \caption{Case 2. $f_B(q)=2^n+1$ and two matching edges are deleted.}
    \label{fig7}
\end{figure}

\item[3.] If $f_B(q)>2^n+1$ then put $x_1=q$ and $x_2=q-1$. Without loss of generality, assume $(x_1,x_2)$ is a non-matching edge, by Proposition~\ref{prop:notunique}. Delete $(x_1,x_2)$ to get a part $X_2=\langle C^1,\dots,C^{x_2}\rangle$. This case is further divided into two subcases 3.1 and 3.2.
		
\item[3.1.] If $\exists (y_1,y_2)$, with $y_1\neq x_1$, on the backbone $B^{x_1}$ of $C^{x_1}$, such that 
\[f_B(x_2)+f_{B^{x_1}}(N^{x_1})-f_{B^{x_1}}(y_1)=2^n,\] 
then delete $(y_1,y_2)$ to get $Y_2=\langle C^{x_1,y_2},\dots,C^{x_1,N^{x_1}}\rangle$. The remaining part is $X=\langle C\backslash X_2\cup Y_2\rangle$. To join the parts $X_2$ and $Y_2$, we add a new edge, as described in the following sub-cases 3.1.1 and 3.1.2.
			
\item[3.1.1.] If $[x_1,y_1]$-path is of odd length, then add the edge $(x_2,y_2)$ to get second part $Y=\langle X_2,(x_2,y_2),Y_2 \rangle$, as shown in Fig.~\ref{fig8}. By the induction hypothesis $1$, $\{x_1,y_1\}$ and $\{x_2,y_2\}$ are mapped to an edge in $\mathcal{Q}_n$ via maps $\phi_1$ and $\phi_2$, respectively. By edge-symmetry of $\mathcal{Q}_n$, we get $\phi_1(x_1)=\phi_2(x_2)$ and $\phi_1(y_1)=\phi_2(y_2)$. Thus, $(\phi(x_1),\phi(x_2))$ and $(\phi(y_1),\phi(y_2))$ are edges in $\mathcal{Q}_{n+1}$.
\begin{figure}
\centering
	\begin{tikzpicture}[scale=0.6]
		\node[circle,fill,draw,inner sep=1] (1) at (-1,0) [label=$1$] {};
		\node[fill,draw,inner sep=2] (q-1) at (1,0) [label=$x_2$] {};
		\node[circle,draw,inner sep=2] (q) at (2,0) [label=$x_1$] {};
		\node[circle,fill,draw,inner sep=1] (N) at (5,0) [label=$N$] {};
		\node[fill,draw,inner sep=2] (y1) at (2,-3) [label=above right:$y_1$] {};
		%\node[circle,fill,draw,inner sep=1] (y2) at (2,-2) [label=right:$y'$] {};
		%\node[circle,fill,draw,inner sep=1] (y3) at (1,-3) [label=$y''$] {};
		\node[circle,draw,inner sep=2] (y4) at (2,-4) [label=right:$y_2$] {};
		\node[circle,fill,draw,inner sep=1] (Nq) at (2,-6) [label=below:$N^{x_1}$] {};
		%\node[circle,fill,draw,inner sep=1] (z1) at (-1,-3) [label=$z_1$] {};
		%\node[circle,fill,draw,inner sep=1] (z2) at (-2,-3) [label=$z_2$] {};
		\node[circle,fill,draw,inner sep=1] (z3) at (6,-3) {};
		
		\draw[dashed] (1)--(q-1); \draw (q-1) to node {/} (q); \draw[dashed] (q) -- (N);
		\draw[dashed] (q) -- (y1);  \draw[dashed] (y1) -- (z3);
		 \draw[dashed] (y4) -- (Nq);
		\draw (y1) to node {---} (y4); 
		
		\draw[thick] (q-1) edge[bend right=40] (y4);
		
		\draw[snake it] (1) -- (-1,-2); \draw[snake it] (q-1) -- (1,-2); \draw[snake it] (N) -- (5,-2); \draw[snake] (Nq) -- (0,-6); \draw[snake] (y4) -- (0,-4);
	\end{tikzpicture}		
	\caption{Case 3.1.1. Two non-matching edges are deleted and one new edge is added to construct two perfectly balanced $2$-caterpillars, each of order $2^n$.}
	\label{fig8}
\end{figure}
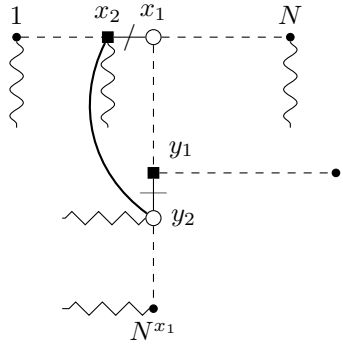

\item[3.1.2.] If $[x_1,y_1]$-path is of even length, then by Proposition \ref{prop:1cat_evenodd} $[y_2,N^{x_1}]$-path is of odd length. By the induction hypothesis $2$, $\{x_1,y_1\}$ is mapped to end vertices of a path $P_3$ in $\mathcal{Q}_n$. As seen in Fig.~\ref{fig9}, we further have two subcases.
\begin{itemize}
    \item[(i)] If $o(C^{x_2})=1$ then add edge $(x_2,N^{x_1})$ to get $Y=\langle X_2,(x_2,N^{x_1}),Y_2\rangle$.
    \item[(ii)] If $o(C^{x_2})>1$ then $[x_2,N^{x_2}]$-path is of odd length (by Proposition \ref{prop:1cat_evenodd}). Add $(N^{x_2},y_2)$ to get $Y=\langle X_2, (N^{x_2},y_2) ,Y_2 \rangle$. 
\end{itemize}
In both the subcases, by the induction hypothesis $2$, $\{x_2,y_2\}$ is mapped to end vertices of a path $P_3$ in $\mathcal{Q}_n$, via map $\phi_2$. By $P_3$-symmetry of $\mathcal{Q}_n$, we see that $\phi_1(x_1)=\phi_2(x_2)$ and $\phi_1(y_1)=\phi_2(y_2)$. Thus, the extended map $\phi$ form edges $(\phi(x_1),\phi(x_2))$ and $(\phi(y_1),\phi(y_2))$ in $\mathcal{Q}_{n+1}$.
	
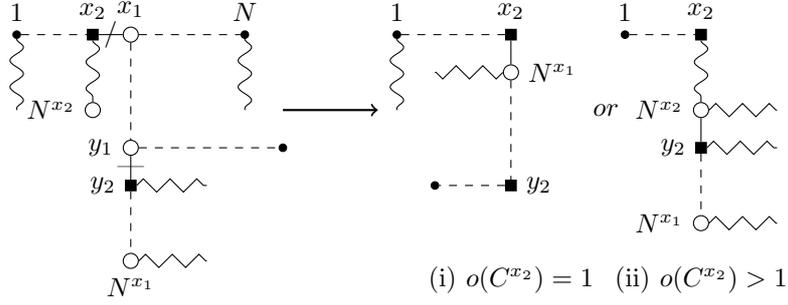
\begin{figure}
\centering
	\begin{tikzpicture}[scale=0.5]
	\begin{scope}[xshift=-5cm]
		\node[circle,fill,draw,inner sep=1] (1) at (-1,0) [label=$1$] {};
		\node[fill,draw,inner sep=2] (q-1) at (1,0) [label=$x_2$] {};
		\node[circle,draw,inner sep=2] (Nx2) at (1,-2) [label=left:$N^{x_2}$] {};
		\node[circle,draw,inner sep=2] (q) at (2,0) [label=$x_1$] {};
		\node[circle,fill,draw,inner sep=1] (N) at (5,0) [label=$N$] {};
		\node[circle,draw,inner sep=2] (y1) at (2,-3) [label=left:$y_1$] {};
		%\node[circle,fill,draw,inner sep=1] (y2) at (2,-2) [label=right:$y'$] {};
		%\node[circle,fill,draw,inner sep=1] (y3) at (1,-3) [label=$y''$] {};
		\node[fill,draw,inner sep=2] (y4) at (2,-4) [label=left:$y_2$] {};
		\node[circle,draw,inner sep=2] (Nq) at (2,-6) [label=below:$N^{x_1}$] {};
		%\node[circle,fill,draw,inner sep=1] (z1) at (-1,-3) [label=$z_1$] {};
		%\node[circle,fill,draw,inner sep=1] (z2) at (-2,-3) [label=$z_2$] {};
		\node[circle,fill,draw,inner sep=1] (z3) at (6,-3) {};
		
		\draw[dashed] (1)--(q-1); \draw (q-1) to node {/} (q); \draw[dashed] (q) -- (N);
		\draw[dashed] (q) -- (y1);  \draw[dashed] (y1) -- (z3);
		 \draw[dashed] (y4) -- (Nq);
		\draw (y1) to node {---} (y4); 
		
		%\draw[densely dashed, thick] (q-1) edge[bend right=40] (Nq);
		
		\draw[snake it] (1) -- (-1,-2);\draw[snake it] (q-1) -- (Nx2); \draw[snake it] (N) -- (5,-2); \draw[snake] (y4) -- (4,-4); \draw[snake] (Nq) -- (4,-6);
	\end{scope}
		
	\begin{scope}[xshift=5cm]
		\node[circle,fill,draw,inner sep=1] (1) at (-1,0) [label=$1$] {};
		\node[fill,draw,inner sep=2] (x2) at (2,0) [label=$x_2$] {};
		\node[circle,draw,inner sep=2] (Nq) at (2,-1) [label=right:$N^{x_1}$] {};
		\node[fill,draw,inner sep=2] (y2) at (2,-4) [label=right:$y_2$] {};
		\node[circle,fill,draw,inner sep=1] (y3) at (0,-4) {};
		
		\node () at (2,-6.5) {(i) $o(C^{x_2})=1$};
		
		\draw[dashed] (1) -- (x2); \draw (x2) -- (Nq); \draw [dashed](Nq) -- (y2);
		\draw[dashed] (y2) -- (y3);
		
		\draw[snake it] (1) -- (-1,-2); \draw[snake] (Nq) -- (0,-1);
		
		\draw[->,thick] (-4,-2) -- (-1.5,-2);
	\end{scope}
	
	\begin{scope}[xshift=10cm]
		\node[circle,fill,draw,inner sep=1] (1) at (0,0) [label=$1$] {};
		\node[fill,draw,inner sep=2] (x2) at (2,0) [label=$x_2$] {};
		\node[circle,draw,inner sep=2] (Nx2) at (2,-2) [label=left:$N^{x_2}$] {};
		\node[fill,draw,inner sep=2] (y2) at (2,-3) [label=left:$y_2$] {};
		\node[circle,draw,inner sep=2] (Nx1) at (2,-5) [label=left:$N^{x_1}$] {};
		
		\node () at (-.5,-2) {$or$};
		
		\draw[dashed] (1) -- (x2); \draw[snake it] (x2) -- (Nx2); \draw (Nx2) -- (y2);
		\draw[dashed] (y2) -- (Nx1); \draw[snake] (y2) -- (4,-3); \draw[snake] (Nx1) -- (4,-5);
		\draw[snake] (Nx2) -- (4,-2);
		\node () at (2,-6.5) {(ii) $o(C^{x_2})>1$};
	\end{scope}
	\end{tikzpicture}	
	\caption{Case 3.1.2.}	
	\label{fig9}
\end{figure}
		
\item[3.2.] If $\exists$ $(y_1,y_2)$ on the backbone of $C^{x_1}$ and $(z_1,z_2)$ on $C^{x_1,y_1}$ such that \[f_B(x_2)+(f_{N^{x_1}}(N^{x_1})-f_{N^{x_1}}(y_1))+(f_{N^{x_1,y_1}}(N^{x_1,y_1})-f_{N^{x_1,y_1}}(z_1))=2^n,\] then these two edges  are unique. Delete $(y_1,y_2)$ and $(z_1,z_2)$ to get the parts $Y_2=\langle C^{x_1,y_2},\dots,C^{x_1,N^{x_1}}\rangle$ and $Z_2=[z_2,N^{x_1,y_1}]$-path. The first part obtained is $X=\langle C\backslash X_2\cup Y_2\cup Z_2\rangle$. Since the degree of $y_1$ is $3$, so there are three possible matching edges it can be incident to, as discussed below.
			
\item[3.2.1.] If $(y_1,y_2) \in  M$ then $(z_1,z_2) \in  M$ and $[y_1,z_1]$-path is of odd length. Add the new matching edge $(y_2,z_2)$ (see Fig.~\ref{fig10}). The non-matching edges along the $[y_1,z_1]$-path become matching edges and vice-versa. Furthermore, 
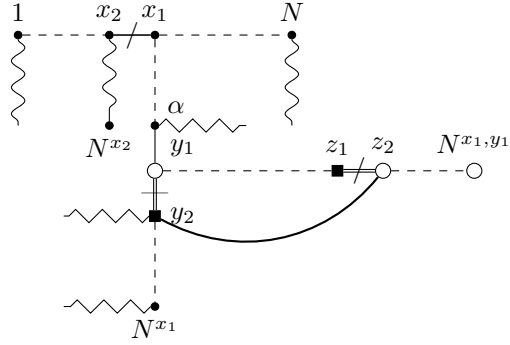
\begin{figure}
\centering
	\begin{tikzpicture}[scale=.6]
		\node[circle,fill,draw,inner sep=1] (1) at (-1,0) [label=$1$] {};
		\node[circle,fill,draw,inner sep=1] (x2) at (1,0) [label=$x_2$] {};
		\node[circle,fill,draw,inner sep=1] (x1) at (2,0) [label=$x_1$] {};
		\node[circle,fill,draw,inner sep=1] (x3) at (1,-2) [label=below:$N^{x_2}$] {};
		\node[circle,fill,draw,inner sep=1] (N) at (5,0) [label=$N$] {};
		\node[circle,draw,inner sep=2] (y1) at (2,-3) [label=above right:$y_1$] {};
		\node[fill,draw,inner sep=2] (y2) at (2,-4) [label=right:$y_2$] {};
		\node[fill,circle,draw,inner sep=1] (Nq) at (2,-6) [label=below:$N^{x_1}$] {};
		\node[fill,draw,inner sep=2] (z1) at (6,-3) [label=$z_1$] {};
		\node[circle,draw,inner sep=2] (z2) at (7,-3) [label=$z_2$]{};
		\node[circle,draw,inner sep=2] (z3) at (9,-3) [label=$N^{x_1,y_1}$] {};
		\node[circle,fill,draw,inner sep=1] (alpha) at (2,-2) [label=above right:$\alpha$] {};
		
		\draw[dashed] (1) -- (x2); \draw (x2) -- (x1); \draw[dashed] (x1) -- (N);
		\draw[dashed] (x1) -- (alpha); \draw (alpha) -- (y1); \draw[double] (y1) to node {---} (y2); \draw[dashed] (y2) -- (Nq); \draw (x2) to node {/} (x1);
		\draw[dashed] (y1) -- (z1); \draw[double] (z1) to node {/} (z2); \draw[dashed] (z2) -- (z3);

		\draw[snake it] (1) -- (-1,-2); \draw[snake it] (x2) -- (x3); \draw[snake it] (N) -- (5,-2);
		\draw[snake] (Nq) -- (0,-6);\draw[snake] (y2) -- (0,-4); \draw[snake] (alpha) -- (4,-2);
		
		\draw[thick] (y2) edge[bend right=40] (z2); 
    \end{tikzpicture}
	\caption{Case 3.2.1. Two non-matching edges $(y_1,y_2)$ and $(z_1,z_2)$ are deleted and compensated by adding one matching edge $(y_2,z_2)$.}
	\label{fig10}
\end{figure}

\item[$(i)$] if $[x_1,y_1]$-path is of even length, add $(x_2,z_2)$ to get the required $2$- caterpillar $Y=\langle X_2,(x_2,z_2),Z_2,(y_2,z_2),Y_2\rangle$ (Fig.~\ref{fig11}). Since the sequence $[x_2,z_2,y_2]$ is $P_3$, so its image $[\phi_2(x_2),\phi_2(z_2),\phi_2(y_2)]$ is $P_3$ in $\mathcal{Q}_n$. By induction hypothesis $2 (ii)$, the sequence $[\phi_1(x_1),\phi_1(z_1),\phi_1(y_1)]$ form $P_3$ in $\mathcal{Q}_n$. By $P_3$-symmetry of $\mathcal{Q}_n$, we get $\phi_1(x_1)=\phi_2(x_2)$, $\phi_1(y_1)=\phi_2(y_2)$ and $\phi_1(z_1)=\phi_2(z_2)$. Thus, the extended map $phi$ form edges $(\phi(x_1),\phi(x_2))$, $(\phi(y_1),\phi(y_2))$ and $(\phi(z_1),\phi(z_2))$ in $\mathcal{Q}_{n+1}$, as required.
				
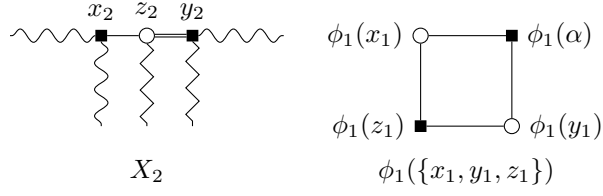
\begin{figure}
\centering

\begin{tikzpicture}[scale=.6]
	\node[fill,draw,inner sep=2] (x2) at (-5,0) [label=$x_2$] {};
	\node[circle,draw,inner sep=2] (z2) at (-4,0) [label=$z_2$] {};
	\node[fill,draw,inner sep=2] (y2) at (-3,0) [label=$y_2$] {};
	
	\node () at (-4,-3) {$X_2$};
	
	\draw[snake it] (x2) -- (-7,0); \draw[snake it] (x2) -- (-5,-2); \draw (x2) -- (z2); \draw[double] (z2) -- (y2); \draw[snake] (y2) -- (-3,-2); \draw[snake it] (y2) -- (-1,0); \draw[snake] (z2) -- (-4,-2);
	
	\node[circle,draw,inner sep=2] (x1) at (2,0) [label=left:$\phi_1(x_1)$] {};
	\node[fill,draw,inner sep=2] (alpha) at (4,0) [label=right:$\phi_1(\alpha)$] {};
	\node[fill,draw,inner sep=2] (z1) at (2,-2) [label=left:$\phi_1(z_1)$] {};
	\node[circle,draw, inner sep=2] (y1) at (4,-2) [label=right:$\phi_1(y_1)$] {};
	
	\node () at (3,-3) {$\phi_1(\{x_1,y_1,z_1\})$};
	
	\draw (x1) -- (alpha) -- (y1) -- (z1) -- (x1);
\end{tikzpicture}
\caption{Case 3.2.1. (i) $\phi_2(\{x_2,z_2,y_2\})$ and $\phi_1(\{x_1,z_1,y_1\})$ are both $K_{1,2}$.}
\label{fig11}
\end{figure}

\item[$(ii)$] if $[x_1,y_1]$-path is of odd length, then, as seen in Fig.~\ref{fig12},
\begin{enumerate}
	\item[$\bullet$] if $o(C^{x_2})>1$ then $[x_2,N^{x_2}]$-path is of odd length (by Proposition \ref{prop:1cat_evenodd}). Add $(N^{x_2},N^{x_1,y_1})$ to get $Y=\langle X_2,(N^{x_2},N^{x_1,y_1}),Z_2,(y_2,z_2),Y_2\rangle$. Here, it is possible that $z_2=N^{x_1y_1}$. 
						
	\item[$\bullet$] if $o(C^{x_2})=1$ then $[y_1,N^{x_1}]$-path is of odd length (by Proposition \ref{prop:1cat_evenodd}). Add $(x_2,N^{x_1})$ to get $Y=\langle X_2,(x_2,N^{x_1}),Y_2,(y_2,z_2),Z_2\rangle$. Here, $y_2=N^{x_1}$ is possible.
\end{enumerate}	
In both the cases, by the induction hypothesis $2$ and since $(y_2,z_2)$ is an edge, the sequence $[\phi_2(x_2),\phi_2(y_2),\phi_2(z_2)]$ form $P_3$ in $\mathcal{Q}_n$. By the induction hypothesis $1 (i)$, the sequence $[\phi_1(x_1),\phi_1(y_1),\phi_1(z_1)]$ form $P_3$ in $\mathcal{Q}_n$. By $P_3$-symmetry of $\mathcal{Q}_n$, we get $\phi_1(x_1)=\phi_2(x_2)$, $\phi_1(y_1)=\phi_2(y_2)$ and $\phi_1(z_1)=\phi_2(z_2)$. Thus, the extended map $phi$ form edges $(\phi(x_1),\phi(x_2))$, $(\phi(y_1),\phi(y_2))$ and $(\phi(z_1),\phi(z_2))$ in $\mathcal{Q}_{n+1}$, as required.

\begin{figure}
\centering
\begin{tikzpicture}[scale=.6]
\begin{scope}
	\node[circle,draw,inner sep=2] (x2) at (0,0) [label=$x_2$] {};
	\node[fill,draw,inner sep=2] (Nx2) at (0,-2) [label=right:$N^{x_2}$] {};
	\node[circle,draw,inner sep=2] (Nqy1) at (0,-3) [label=right:$N^{x_1y_1}$] {};
	\node[circle,draw,inner sep=2] (z2) at (0,-5) [label=right:$z_2$] {};
	\node[fill,draw,inner sep=2] (y2) at (0,-6) [label=right:$y_2$] {};
	
	\node () at (0,-8.5) {$C^{x_2}>1$};
	
	\draw[dashed] (x2) -- (Nx2); \draw (Nx2)--(Nqy1); \draw[dashed] (Nqy1) -- (z2); \draw[double] (z2)-- (y2);
	 \draw[snake it] (x2) -- (-2,0); \draw[snake] (Nx2) -- (-2,-2); \draw[snake] (y2) -- (-2,-6); \draw[snake it] (y2) -- (0,-8);
\end{scope}

\begin{scope}[xshift=8cm]
	 \node[circle,draw,inner sep=2] (x2) at (0,0) [label=$x_2$] {};
	\node[fill,draw,inner sep=2] (Nq) at (0,-1) [label=right:$N^{x_1}$] {};
	\node[fill,draw,inner sep=2] (y2) at (0,-3) [label=right:$y_2$] {};
	\node[circle,draw,inner sep=2] (z2) at (0,-4) [label=right:$z_2$] {};
	
	\node () at (0,-5.5) {$C^{x_2}=1$};
	
	\draw (x2) -- (Nq); \draw[dashed] (Nq)--(y2); \draw[double] (y2) -- (z2); 
	
	\draw[snake it] (x2) -- (-2,0); \draw[snake] (Nq) -- (-2,-1); \draw[snake] (y2) -- (-2,-3); \draw[snake] (z2) -- (-2,-4);
\end{scope}
\end{tikzpicture}
\caption{Case 3.2.1. (ii).}
\label{fig12}
\end{figure}
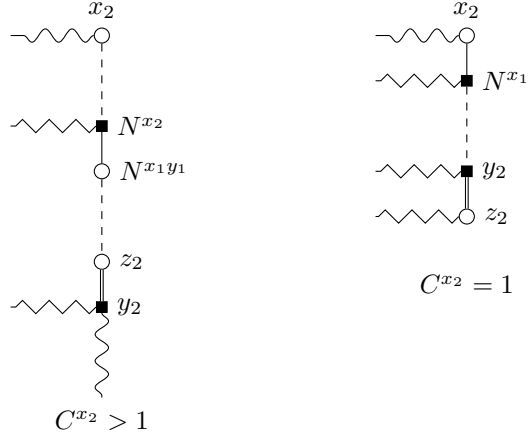								
%In both the cases, $\phi_2(\{x_2,y_2,z_2\})$ form $K_{1,2}$.
	
\item[3.2.2.] If $(y_1,y')$ is a matching edge on $C^{x_1,y_1}$ then  $[y_1,z_1]$-path is of odd length. Also $[y_2,N^{x_1}]$-path is of odd length (by Proposition \ref{prop:1cat_evenodd}). Add $(y_2,z_2)$. As seen in Fig.~\ref{fig13},
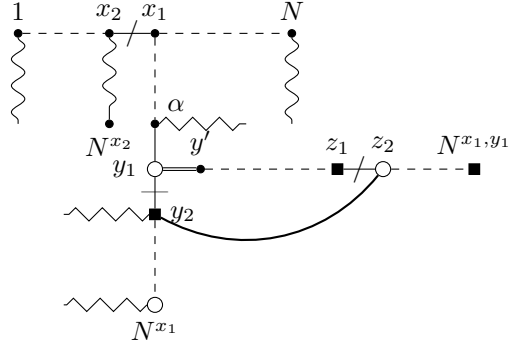
\begin{figure}
\centering

\begin{tikzpicture}[scale=.6]
	\node[circle,fill,draw,inner sep=1] (1) at (-1,0) [label=$1$] {};
		\node[circle,fill,draw,inner sep=1] (x2) at (1,0) [label=$x_2$] {};
		\node[circle,fill,draw,inner sep=1] (x3) at (1,-2) [label=below:$N^{x_2}$] {};
		\node[circle,fill,draw,inner sep=1] (x1) at (2,0) [label=$x_1$] {};
		\node[circle,fill,draw,inner sep=1] (N) at (5,0) [label=$N$] {};
		\node[circle,draw,inner sep=2] (y1) at (2,-3) [label=left:$y_1$] {};
		\node[fill,draw,inner sep=2] (y2) at (2,-4) [label=right:$y_2$] {};
		\node[circle,fill,draw,inner sep=1] (y') at (3,-3) [label=$y'$] {};
		\node[circle,draw,inner sep=2] (Nq) at (2,-6) [label=below:$N^{x_1}$] {};
		\node[fill,draw,inner sep=2] (z1) at (6,-3) [label=$z_1$] {};
		\node[circle,draw,inner sep=2] (z2) at (7,-3) [label=$z_2$]{};
		\node[fill,draw,inner sep=2] (z3) at (9,-3) [label=$N^{x_1,y_1}$] {};
		\node[circle,fill,draw,inner sep=1] (alpha) at (2,-2) [label=above right:$\alpha$] {};
		
		\draw[dashed] (1) -- (x2); \draw (x2) to node {/} (x1); \draw[dashed] (x1) -- (N);
		\draw[dashed] (x1) -- (alpha); \draw (alpha) -- (y1); \draw (y1) to node {---} (y2); \draw[dashed] (y2) -- (Nq);
		\draw[dashed] (y') -- (z1); \draw (z1) to node {/} (z2); \draw[dashed] (z2) -- (z3);
		\draw[double] (y1) -- (y');
		
		\draw[snake it] (1) -- (-1,-2); \draw[snake it] (x2) -- (x3); \draw[snake it] (N) -- (5,-2);
		\draw[snake] (y2) -- (0,-4); \draw[snake] (alpha) -- (4,-2); \draw[snake] (Nq) -- (0,-6);
		
		\draw[thick] (y2) edge[bend right=40] (z2);
\end{tikzpicture}
\caption{Case 3.2.2.}
\label{fig13}
\end{figure}

\item[(i)] If $[x_1,y_1]$-path is of even length then add $(x_2,z_2)$ to get the second part $Y=\langle X_2,(x_2,z_2),Z_2,(y_2,z_2),Y_2\rangle$.
\item[(ii)] If $[x_1,y_1]$-path is of odd length then,
	\begin{enumerate}
		\item[$\bullet$] if $s(C^{x_2})>1$ then $[x_2,N^{x_2}]$-path is of odd length (by Proposition \ref{prop:1cat_evenodd}). Add $(N^{x_2},N^{x_1})$ to get $Y=\langle X_2,(N^{x_2},N^{x_1}),Y_2,(y_2,z_2),Z_2 \rangle$.  
		\item[$\bullet$] if $s(C^{x_2})=1$ and since $[z_2,N^{x_1,y_1}]$-path is of odd length, $(x_2,N^{x_1,y_1})$ is added to get $Y=\langle X_2,(x_2,N^{x_1,y_1}),Y_2,(y_2,z_2),Z_2\rangle$.
    \end{enumerate} 
	
\item[3.2.3.] If $(y',y_1)\in  M$ on the backbone of $C^{x_1}$, with $y'\neq y_2$, then $[y_1,z_1]$-path and $[z_2,N^{x_1,y_1}]$-path are both of even length. Add $(y_2,N^{x_1,y_1})$ (see Fig~\ref{fig14}). We further discuss two sub-cases.

\begin{figure}
\centering
	\begin{tikzpicture}[scale=.6]
		\node[circle,fill,draw,inner sep=1] (1) at (-1,0) [label=$1$] {};
			\node[circle,fill,draw,inner sep=1] (x2) at (1,0) [label=$x_2$] {};
			\node[circle,fill,draw,inner sep=1] (x3) at (1,-2) [label=below:$N^{x_2}$] {};
			\node[circle,fill,draw,inner sep=1] (x1) at (2,0) [label=$x_1$] {};
			\node[circle,fill,draw,inner sep=1] (N) at (5,0) [label=$N$] {};
			\node[circle,draw,inner sep=2] (y1) at (2,-3) [label=above right:$y_1$] {};
			\node[fill,draw,inner sep=2] (y2) at (2,-4) [label=right:$y_2$] {};
			\node[circle,fill,draw,inner sep=1] (y') at (2,-2) [label=above right:$y'$] {};
			\node[circle,fill,draw,inner sep=1] (Nq) at (2,-6) [label=below:$N^{x_1}$] {};
			\node[circle,draw,inner sep=2] (z1) at (6,-3) [label=$z_1$] {};
			\node[fill,draw,inner sep=2] (z2) at (7,-3) [label=$z_2$]{};
			\node[circle,draw,inner sep=2] (z3) at (9,-3) [label=$N^{x_1,y_1}$] {};
			
			\draw[dashed] (1) -- (x2); \draw (x2) -- (x1); \draw[dashed] (x1) -- (N);
			\draw[dashed] (x1) -- (y'); \draw (y1) to node {---} (y2); \draw[dashed] (y2) -- (Nq);
			\draw[double] (y') -- (y1); \draw (z1) to node {/} (z2); \draw[dashed] (z2) -- (z3);
			\draw[dashed] (y1) -- (z1);
			
			\draw[snake it] (1) -- (-1,-2); \draw[snake it] (x2) -- (x3); \draw[snake it] (N) -- (5,-2);
			\draw[snake] (y2) -- (0,-4);\draw[snake] (y') -- (4,-2); \draw[snake] (Nq) -- (0,-6);
			
			\draw[thick] (y2) edge[bend right=40] (z3);
	\end{tikzpicture}
	\caption{Case 3.2.3.}
	\label{fig14}
\end{figure}
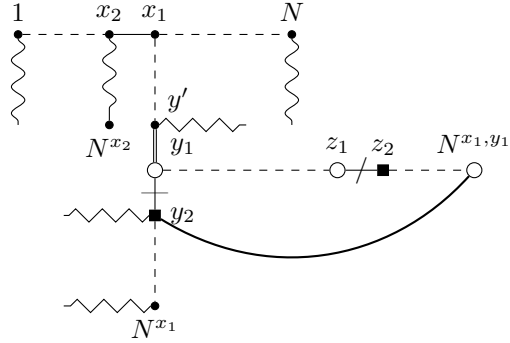
	
\item[(i)] If $[x_1,y_1]$-path is of even length then, as seen in the Fig.~\ref{fig15},
\begin{itemize}
\item[$\bullet$] if $o(C^{x_2})>1$ then $[x_2,N^{x_2}]$-path is of odd length (by Proposition \ref{prop:1cat_evenodd}). By Proposition \ref{prop:1cat_evenodd}, $[y_2,N^{x_1}]$-path is of even length. If $o(Y_2)>2$ then add $(N^{x_2},N^{x_1})$ to get   $Y=\langle X_2,(N^{x_2},N^{x_1}),Y_2,(y_2,N^{x_1,y_1}),Z_2\rangle$, otherwise add $(N^{x_2},y_2)$ to get $Y=\langle X_2,(N^{x_2},y_2),Y_2,(y_2,N^{x_1,y_1}),Z_2\rangle$.
\item[$\bullet$] if $o(C^{x_2})=1$ then add $(x_2,N^{x_1})$, since $[y_2,N^{x_1}]$-path is of odd length by Proposition \ref{prop:1cat_evenodd}, to get $Y=\langle X_2,(x_2,N^{x_1}),Y_2,(y_2,N^{x_1,y_1}),Z_2\rangle$.
\end{itemize}
		
\item[(ii)] If $[x_1,y_1]$-path is of odd length then, 
	\begin{itemize}
		\item[$\bullet$] If $o(C^{x_2})>1$ then $[x_2,N^{x_2}]$-path is of odd length (by Proposition \ref{prop:1cat_evenodd}). Add $(N^{x_2},N^{x_1})$ to get $Y=\langle X_2,(N^{x_2},N^{x_1}),Y_2,(y_2,N^{x_1,y_1}),Z_2\rangle$. 
		\item[$\bullet$] If $o(C^{x_2})=1$ then add $(x_2,z_2)$ to get $Y=\langle X_2,(x_2,z_2),Z_2,(y_2,N^{x_1,y_1}),\linebreak Y_2\rangle$. 
	\end{itemize}	
	
\begin{figure}
\centering
\begin{tikzpicture}[scale=.6]
	\begin{scope}
		\node[fill,draw,inner sep=2] (x2) at (0,3) [label=$x_2$] {};
		\node[circle,draw,inner sep=2] (Nx2) at (0,1) [label=right:$N^{x_2}$] {};
		\node[fill,draw,inner sep=2] (Nx1) at (0,0) [label=right:$N^{x_1}$] {};
		\node[fill,draw,inner sep=2] (y2) at (0,-2) [label=right:$y_2$] {};
		\node[circle,draw,inner sep=2] (Nx1y1) at (-1,-2) [label=$N^{x_1,y_1}$] {};
		\node[fill,draw,inner sep=2] (z2) at (-3,-2) [label=$z_2$] {};
		
		\node () at (0,-4.5) {$s(C^{x_2})>1$};
		
		\draw[snake it] (x2) -- (Nx2); \draw (Nx2) -- (Nx1); \draw[snake] (Nx2) -- (-2,1);
	    \draw (y2) -- (Nx1y1); \draw[snake it] (x2) -- (-3,3);
		\draw[dashed] (Nx1) -- (y2); \draw[dashed] (z2) -- (Nx1y1); \draw[snake] (Nx1) -- (-2,0);
		\draw[snake] (y2) -- (0,-4);
	\end{scope}
	\begin{scope}[xshift=6cm]
		\node[fill,draw,inner sep=2] (x2) at (0,1) [label=$x_2$] {};
		
		\node[circle,draw,inner sep=2] (Nx1) at (0,0) [label=right:$N^{x_1}$] {};
		\node[fill,draw,inner sep=2] (y2) at (0,-2) [label=right:$y_2$] {};
		\node[circle,draw,inner sep=2] (Nx1y1) at (-1,-2) [label=$N^{x_1,y_1}$] {};
		\node[fill,draw,inner sep=2] (z2) at (-3,-2) [label=$z_2$] {};
		
		\node () at (0,-4.5) {$s(C^{x_2})=1$};				
		
		\draw (x2) -- (Nx1); \draw (y2) -- (Nx1y1); \draw[snake it] (x2) -- (-3,1);
		\draw[dashed] (Nx1) -- (y2); \draw[dashed] (z2) -- (Nx1y1); \draw[snake] (Nx1) -- (-2,0);
		\draw[snake] (y2) -- (0,-4);
	\end{scope}
	\begin{scope}[xshift=12cm,yshift=0.5cm]
		\node[circle,draw,inner sep=2] (Nx1) at (0,0) [label=right:$\phi_2(\alpha)$] {};
		\node[fill,draw,inner sep=2] (x2) at (0,-1.4) [label=below:$\phi_2(x_2)$] {};
		\node[fill,draw,inner sep=2] (y2) at (-1,1) [label=$\phi_2(y_2)$] {};
		\node[fill,draw,inner sep=2] (z2) at (1,1) [label=$\phi_2(z_2)$] {};

		\draw (Nx1)--(x2); \draw (Nx1)--(y2);\draw (Nx1)--(z2); 
	\end{scope}
	\node () at (2,0) {\textbf{or}};
	\draw[->,thick] (6.5,0.5) -- (9,0.5);
\end{tikzpicture}		
\caption{Case 2.2.3. (i).}
\label{fig15}
\end{figure}
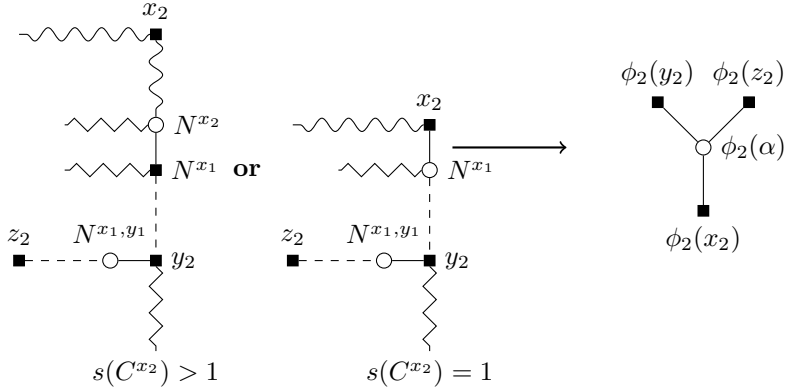				
\end{enumerate}
\end{proof}

\section{Conclusion}
As a next improvement of the result presented in this paper, we can consider embedding perfectly balanced $k$-caterpillars ($k\geq 3$). However, the proof technique of creating sub-caterpillars may not be appropriate for these caterpillars. This is because, we will have to then delete more than $3$ edges from the $k$-caterpillars to obtain a perfectly balanced sub-caterpillar. But then, there do not exist any more path symmetries in $\mathcal{Q}_n$ to embed the sub-caterpillar as desired. 

An adaption of this technique used in this paper to embed equibipartite $2$-caterpillars is promising and is being attempted.

\bibliographystyle{acm}
\bibliography{EmbedPerfectBalanced2CatIntoOptHypercube}

\begin{thebibliography}{1}

\bibitem{bezrukov1998}
{\sc Bezrukov, S., Monien, B., Unger, W., and Wechsung, G.}
\newblock Embedding ladders and caterpillars into the hypercube.
\newblock {\em Discrete Applied Mathematics 83}, 1 (1998), 21--29.

\bibitem{choudum2016}
{\sc Choudum, S., Sivkumar, L., and Sunitha, V.}
\newblock Graph embedding and interconnection networks.
\newblock In {\em Handbook of Graph Theory, Combinatorial Optimization, and
  Algorithms}. CRC Press, 2015, ch.~26, pp.~653--688.

\bibitem{havel1984}
{\sc Havel, I.}
\newblock On hamiltonian circuits and spanning trees of hypercubes.
\newblock {\em \v{C}asopis pro p\v{e}stov\'{a}n\'{i} matematiky 109}, 2 (1984),
  135--152.

\bibitem{havel1986}
{\sc Havel, I.~M., and Liebl, P.}
\newblock One-legged caterpillars span hypercubes.
\newblock {\em J. Graph Theory 10\/} (1986), 69--77.

\bibitem{monien2018}
{\sc Monien, B., and Wechsung, G.}
\newblock Balanced caterpillars of maximum degree 3 and with hairs of arbitrary
  length are subgraphs of their optimal hypercube.
\newblock {\em Journal of Graph Theory 87}, 4 (2018), 561--580.

\end{thebibliography}
\end{document}